\documentclass[10pt]{article}

\usepackage{amssymb,amsmath,amsthm}
\usepackage{geometry}
\usepackage[T1]{fontenc}    
\usepackage[utf8]{inputenc}

\newcommand{\C}{\mathcal{C}}

\newcommand{\R}{\mathbb{R}}

\newcommand{\cF}{\mathcal{F}}
\newcommand{\cG}{\mathcal{G}}

\newcommand{\fa}{\mathfrak{a}}

\newcommand{\fg}{\mathfrak{g}}

\newcommand{\fk}{\mathfrak{k}}

\newcommand{\fp}{\mathfrak{p}}

\newcommand{\G}{\Gamma}
\newcommand{\g}{\gamma}

\newcommand{\tv}{\rightarrow}

\DeclareMathOperator{\SL}{SL}

\DeclareMathOperator{\Card}{Card}

\newtheorem{theorem}{Theorem}[section]
\newtheorem{definition}[theorem]{Definition}
\newtheorem{lemma}[theorem]{Lemma}
\newtheorem{proposition}[theorem]{Proposition}
\newtheorem{corollary}[theorem]{Corollary}

\usepackage{color}

\DeclareMathOperator{\Id}{Id}

\renewcommand{\phi}{\varphi}
\renewcommand{\epsilon}{\varepsilon}

\author{Olivier Glorieux\footnote{This project received funding from the European Research Council (ERC) under the European Union's Horizon 2020 research and innovation programme (ERC starting grant DiGGeS, grant agreement No 715982).}, Samuel Tapie\footnote{S. Tapie acknowledges support from A.N.R. grant CCEM (ANR-17-CE40-0034) and from the Centre
Henri Lebesgue ANR-11-LABX- 0020-01}}

\title{Critical exponents of normal subgroups in higher rank}

\makeindex
\begin{document}
\maketitle

\begin{abstract}

We study the critical exponents of discrete subgroups of a higher rank semi-simple real linear Lie group $G$. Let us fix a Cartan subspace $\mathfrak a\subset \mathfrak g$ of the Lie algebra of $G$. We show that if $\Gamma< G$ is a discrete group, and $\Gamma' \triangleleft \Gamma$ is a Zariski dense \emph{normal} subgroup, then the limit cones of $\Gamma$ and $\Gamma'$ in $\mathfrak a$ coincide. Moreover, for all linear form $\phi : \mathfrak a\to \mathbb R$ positive on this limit cone, the critical exponents in the direction of $\phi$ satisfy $\displaystyle \delta_\phi(\Gamma') \geq \frac 1 2 \delta_\phi(\Gamma)$. Eventually, we show that if $\Gamma'\backslash \Gamma$ is amenable, these critical exponents coincide.

\end{abstract}

\section{Introduction}

Let $G$ be a real linear semi-simple Lie group, $K<G$ a maximal compact subgroup and $X = G/K$ the associated Riemmannian symmetric space, whose distance is denoted by $d$. We write $o\in X$ the fixed point of $K$. The sectional curvatures on $X$ are non-positive. When $G$ has rank one, we can scale the metric such that  the sectionnal curvatures on $X$ are at most $-1$. 

Let $\Gamma < G$ be a discrete subgroup. To study the growth of $\Gamma$, a central quantity is its \emph{critical exponent}, defined by
$$
\delta(\Gamma) = \limsup_{R\to +\infty} \frac 1 R \log \{ \gamma \in \Gamma \; ; \; d(o, \gamma o) \leq R\}.
$$

Since the sectional curvatures of $X$ are bounded from below, it follows from the Bishop-Gromov comparison theorem  that this critical exponent is finite, bounded from above by the \emph{volume entropy} of $G$ which is defined by
$$
h(G) = \limsup_{R\to +\infty} \frac 1 R \log {\rm Vol}(B(o,R)).
$$

The critical exponent has been widely studied since the 70's for Riemannian manifolds with curvature at most $-1$ (and CAT$(-1)$ metric spaces). In this case, it coincides with the topological entropy of the geodesic flow of $\Gamma \backslash X$. In this paper, we will be interested in the following elementary question. 

\begin{center}
{\bf If $\Gamma'\triangleleft \Gamma$ is a normal subgroup, how are related $\delta(\Gamma')$ and $\delta(\Gamma)$ ? }
\end{center}

\subsection{Normal subgroups, limit cones and lower bound on critical exponents}

By construction, $\delta(\Gamma') \leq \delta(\Gamma)$ for all subgroups $\Gamma'<\Gamma$. When $\Gamma'$ is a \emph{normal} subgroup of $\Gamma$, the following results have been shown in \cite{roblin2005fatou}.

\begin{theorem}[\cite{roblin2005fatou}, Theorem 2.2.1]\label{th_roblin-delta1>2delta0}
Let $X$ be a CAT(-1) metric space. Let $\Gamma\subset Isom(X)$ be a discrete group and $\Gamma'\triangleleft \Gamma$ be a non elementary \emph{normal} subgroup of $\G_0$. Then their critical exponents satisfy $\displaystyle \delta(\Gamma') \geq \frac{1}{2}\delta(\Gamma)
$.
\end{theorem}

This result may be surprising due to the following example. In the setting of Theorem \ref{th_roblin-delta1>2delta0}, let $g,h\in \Gamma$ be hyperbolic elements whose fixed points on the visual boundary $\partial X$ are all distinct. For all $n\in \mathbb N$, let us consider the subgroup $\Gamma_n = \langle g^n, h^n\rangle <\Gamma$ generated by $g^n$ and $h^n$. For $n$ large enough, the group $\Gamma_n$ has a structure of \emph{Schottky group}. It is a free group generated by $g^n$ and $h^n$, and it can easily be shown that $\displaystyle \lim_{n\to +\infty} \delta(\Gamma_n) = 0$. Nevertheless, writing $\Gamma'_n$ for the \emph{normal closure} of $\Gamma_n$ in $\Gamma$, it follows from Theorem \ref{th_roblin-delta1>2delta0} that $\delta(\Gamma'_n)$ cannot be smaller than $\frac 1 2 \delta(\Gamma)$. We will show that a similar phenomenon holds in higher rank symmetric spaces.

\medskip

On non-positively curved manifolds, the study of the critical exponent has been of more recent interest. For discrete groups acting on higher rank symmetric spaces, it has been known since the PhD thesis of J.F. Quint (cf \cite{quint2002mesures}) that it is interesting to consider \emph{directional critical exponents} which we define now. We refer to Section \ref{sec-background} for the geometric background.

\medskip

Let us fix a Cartan subspace $\mathfrak a\subset \mathfrak g$ of the Lie algebra of $G$ and a Weyl chamber $\mathfrak{a^+}\subset \mathfrak a$. Let $C(\Gamma) \subset \overline{\mathfrak a^+}$ be the limit cone of $\Gamma$. This cone has been introduced by Y. Benoist in \cite{benoist1997proprietes}, and describes the directions in the Weyl chamber in which the group grows. By definition, if $\Gamma'<\Gamma$, the limit cones satisfy $C(\Gamma') \subset C(\Gamma)$.

For any linear form $\phi\in \mathcal L(\mathfrak a, \mathbb R)$, which we suppose to be positive on $C(\Gamma)$, let us write
$$
\delta_\varphi(\Gamma) := \limsup_{R\tv \infty} \frac{1}{R} \log \Card \{ \g \in \Gamma\, |\, \varphi (\kappa(\g)) \leq R \},
$$
where $\kappa(\g)$ denotes the Cartan projection of $\g$, see Section \ref{sec-cartan projection}.

Our first result is the following.

\begin{theorem}\label{th-main limit cone}
Let $G$ be a real linear,  semisimple, connected, Lie group with finite center. 
Let $\Gamma$ be a discrete subgroup of $G$ and $\Gamma'\triangleleft \Gamma$ be a  normal subgroup. If $\Gamma'$ is Zariski dense in $G$, then the two limit cones coincide: $C(\Gamma)=C(\Gamma')$.

Moreover, for all linear forms $\varphi\in \mathcal L(\mathfrak a, \mathbb R)$ {which are positive on $C(\G)$}, 
 we have $\displaystyle \delta_\varphi(\G')\geq \frac{1}{2}\delta_\varphi(\G)$.
\end{theorem}

We have separated the statements to emphasize on the equality of limit cones, which seems unknown. Since a linear form which vanishes on some {open subset} of $C(\Gamma)$ has infinite critical exponent, this equality of limit cones follows from the inequality $\delta_\varphi(\G')\geq \frac{1}{2}\delta_\varphi(\G)$ for all $\phi$ which are positive on $C(\Gamma')$.

The Zariski dense assumption on $\G'$ is necessary to avoid trivial counter-examples as $\G'=\G_\ell\times \{\Id\} \subset \G_\ell \times \G_r =\G$ for discrete groups in a product, for which Theorem \ref{th-main limit cone} is trivially false. Our approach is based on the existence of $\phi$-conformal Patterson-Sullivan densities for $\Gamma'$, which are not defined for discrete groups such as $\G_\ell\times \{\Id\} \subset \G_\ell \times \G_r$.

{
Let us say that the group $\G\subset G $ is  \emph{normally irreducible} if it intersects trivially all normal subgroups $N \triangleleft G$ when $N\neq G$. We will show in Proposition \ref{prop-CNS normally irred} that a Zariski dense subgroup $\Gamma < G$ is normally irreducible if and only if all its non-trivial normal subgroups are Zariski dense
\begin{corollary}
Let $G$ be a real linear,  semisimple, connected, Lie group with finite center. 
Let $\Gamma$ be a normally irreducible Zariski dense discrete subgroup of $G$. Then for all non trivial normal subgroup $\Gamma'\triangleleft \Gamma$, the limit cones $C(\Gamma)$ and $C(\Gamma')$ coincide. Moreover, for all linear forms $\varphi\in \mathcal L(\mathfrak a, \mathbb R)$ which are positive on $C(\G)$, we have $\displaystyle \delta_\varphi(\G')\geq \frac{1}{2}\delta_\varphi(\G)$.
\end{corollary}

%
%
%

}

\subsection{Equality of critical exponents for co-amenable normal subgroups}

Once the equality of the limit cones has been proven, it is natural to look for conditions which ensure the equality between the critical exponents for $\G$ and $\G'$. When both groups are lattices, the question is settled by a famous theorem of Leuzinger.

\begin{theorem}[Leuzinger, \cite{leuzinger2003kazhdan}]
Let $\G$ be a (possibly non-uniform) lattice in $G$, and let $h(G)$ be the volume entropy of $G$. Then 
$\delta(\G) = h(G)$ if and only if $\G$ is a lattice. 
\end{theorem}

We are hence interested in the situation when $\Gamma'$ is not a lattice. In negative curvature, the following has been shown by Roblin.

\begin{theorem}[\cite{roblin2005fatou}, Theorem 2.2.2]\label{th_roblin-amenable}
Let $X$ be a CAT(-1) metric space. Let $\Gamma\subset Isom(X)$ be a non-elementary discrete group. Let $\Gamma'\triangleleft \Gamma$ be a \emph{normal} subgroup of $\G$ such that $\Gamma' \backslash \Gamma$ is amenable. Then $\delta(\Gamma') =\delta(\Gamma)$.
\end{theorem}

We will define \emph{amenable groups} in Section \ref{sec -amenability and critical exponent}. Our second main result extends this theorem to discrete groups of higher rank symmetric spaces. 

\begin{theorem}\label{th-main critical exponent}
Let $G$ be a real linear,  semisimple, connected, Lie group with finite center. 
Let $\G$ be a discrete subgroup of $G$. Let $\G'$ be a normal, Zariski dense subgroup of $\G$. If $\G'\backslash \G$ is amenable then for all linear form $\varphi\in \mathcal L(\fa, \R)$ {which are positive on $C(\G)$}, the associated critical exponents satisfy $\displaystyle \delta_\varphi(\G') =\delta_\varphi (\G)$.
\end{theorem}

Let us point that, in our proof as in Roblin's, the \emph{normal} assumption is crucial. Removing this assumption is possible in negative curvature, but requires an additional hypothesis on $\Gamma$. This hypothesis also allows to have a full characterization of the equality of critical exponents by amenability, which has been shown by Coulon, Dougall, Schapira and Tapie.

\begin{theorem}[\cite{CDST}, Theorem 1.1]\label{th:CDST}
Let $\G$ Let X be a proper Gromov-hyperbolic geodesic space.  Let $\G$ be a discrete group acting properly by isometries on $X$, and $\G'$ a subgroup of $\G$. Assume that the action of $\G$ is strongly positively recurrent. The following are equivalent.
\begin{itemize}
\item $ \delta(\G)=\delta(\G')$
\item The subgroup $\G'$ is co-amenable in $\G$.
\end{itemize}
\end{theorem}

Being \emph{strongly positively recurrent} means that for some suitable notion of \emph{critical exponent at infinity} $\delta^\infty(\Gamma)$, we have $\delta(\Gamma)>\delta^\infty(\Gamma)$. This property is also called a \emph{critical gap}. In negative curvature, this has strong implications on the dynamics of the group $\Gamma$ acting on $X$, and on the geodesic flow of the quotient $X/\Gamma$. We refer the reader to \cite{ST19}, \cite{CDST}. Studying the notion of strongly positively recurrent discrete groups in higher rank symmetric spaces will be done elsewhere.

\subsection{Organisation of the paper}
In the next section, we recall the necessary background on dynamics of discrete groups in symmetric spaces.

In Section \ref{sec-subgroup and critical exponent} we present the proof of Theorem \ref{th-main limit cone}. It mostly relies on adapting the ideas from \cite{roblin2005fatou} to higher rank setting. The main idea is to prove that the orbits of the group $\G$ satisfies a so-called \emph{shadow principle} with respect to the conformal densities associated to $\G'$, see Section \ref{sec-extending the shadow lemma}.

In the last section, we introduce amenable groups and prove Theorem \ref{th-main critical exponent}. It also relies on the shadow principle, combined with the use of amenability of $\G' \backslash\G$ to construct an average of the $\G$-equivariant conformal density, see Section \ref{sec-averaging conformal density}, from which we will control the critical exponents.

\section{Background on higher rank symmetric spaces}\label{sec-background}
In this section we first recall some classical facts on the geometry of higher rank symmetric spaces. We then describe some basic facts about discrete isometry groups of such spaces, and we introduce $\varphi$-conformal densities on the Furstenberg boundary which will be the main tools in the proof of our results. 
Our main references for this section are  \cite{quint2002mesures,thirion,deykapovich}. {Similar results under a slightly different point of view are presented in \cite{Lin04}.}

\subsection{Symmetric spaces}\label{sec-symmetric spaces}
We fix once for all a semisimple real linear Lie group $G$.
In this section we present some geometric properties of the symmetric space associated to $G$ which we will use.  

\subsubsection{Structure of semisimple Lie groups} 
Let $K$ be a maximal compact subgroup of $G$. We denote by $X= G/K$ the associated symmetric riemannian space.

The Lie algebra of $G$ (respectively $K$) is denoted by $\fg$ (respectively $\fk$). Let $\fg = \fk \oplus \fp$ be a Cartan decomposition of $\fg$.

The maximal abelian Lie algebra of $\fp$ is called a Cartan subspace and is denoted by $\fa$. Geometrically, $\fa$ corresponds via the exponential map to an isometric copy of a Euclidean space in $X$ of maximal dimension. Any isometric, totally geodesic copy of $\R^n$  is called a flat of $X$ of dimension $n$. The dimension of $\fa$ is called \emph{the rank} of $G$. By definition, when this rank is at least $2$, there exists flats of dimension greater than $2$ along which the sectional curvatures vanishes.

The space $\fg$ then decomposes into a direct sum $\fg = \oplus_{\alpha\in \Sigma} \fg_\alpha$, where the sum is taken over a set of linear forms on $\fa$, called roots, for which 
$$\fg_\alpha := \{ u \in \fg\, |\, [a,u] =\alpha(a) u, \, \forall a\in \fa\},$$
is non trivial.

The kernels of non zero roots cut the vector space $\fa$ into connected  components, called Weyl chambers. We choose one, which we call the positive Weyl chamber and denote by $\fa^+$.

\subsubsection{Cartan projection}\label{sec-cartan projection}

 The main difference between the hyperbolic setting and  higher rank symmetric spaces is the existence of such flats, for which the sectional curvature is zero.  Directions in flats provide invariants for the dynamics of the group $G$ and hence for the geodesic flow on $X$. We now present how we can take this into account.

The \emph{Cartan projection}, which generalizes the polar decomposition for any symmetric space, is defined as follows.

\begin{theorem}
For every $g\in G$  there exists $k_1, k_2 \in K$ and a unique $\kappa(g) \in \overline{\fa^+}$  such that
$$g = k_1 \exp(\kappa(g) ) k_2.$$
The element $\kappa(g)$ is called the Cartan projection of $g$. 
\end{theorem}
  
We refer the reader to \cite{Helgason}, Chap. VI. for a proof. The Killing form on $G$, which induces the metric on $X= G/K$ and in particular gives a scalar product on $\fp$ which finally restrict to $\fa$. The norm of $\kappa(g)$ for this scalar product is by construction the distance $d(g [K], [K]) $ in $X$. For all $x,y \in X$ let $g_x, g_y\in G$ be such that $x=g_xK$ and $y=g_yK$. Then the distance $d(x,y)$ in $X$ is given by $\|\kappa( {g_y}^{-1} g_x)\| $. 

In order to take into account the direction inside $\fa^+$,  we introduce a  $\fa^+$-valued ``distance'' defined by:
$$a (x,y) := \kappa( {g_y}^{-1} g_x). $$

The following lemma will be used in the sequel as one commonly uses the triangle inequality in negative curvature:
\begin{lemma}\cite[Lemme 4.6]{benoist1997proprietes}\label{lemma-triangle inequality in higher rank}
For all compact subset $L$ of $G$, there exists a positive real number $M>0$ such that for all $\ell_1,\ell_2 \in L$, and for all $x,y\in X$:
$$\| a(\ell_1 x,  \ell_2 y  ) -a(x,y)\| \leq M.$$
\end{lemma}

\subsubsection{Furstenberg boundary}

When the rank of $X$  is at least $2$, the isometry group of $X$ does not act transitively on the unitary tangent bundle. This difference with strictly negative curvature spaces makes the use of the geometric boundary (ie. equivalence classes of asymptotic rays)  less relevant. One will instead consider classes of asymptotic Weyl chambers, which we introduce now. 

A \emph{geometric Weyl chamber} is by definition a map $ w : \fa^+ \tv X$ of the form 
$$w(a) = g\cdot \exp( a), $$
where $g$ is an element of $G$. 
Two geometric Weyl chambers  $w_1, w_2$ are said \emph{asymptotic} if 
$$\sup_{a \in \fa} d(w_1(a), w_2(a) ) <+\infty.$$
We denote by $\cF$ the space of classes of asymptotic geometric Weyl chambers. 
By definition, the group $G$ acts transitively on $\cF$. The stabilizer of the classes corresponding to $g=\Id$ is denoted by $P$. One has therefore the identification $\cF \simeq G/P$. The space $\cF$ is called the \emph{Furstenberg boundary} of $X$. 

In higher rank symmetric spaces, geometric Weyl chambers play a role analogous to geodesic rays in rank one. For a geometric Weyl chamber $w$, call $w(0)$ its origin. For any class of asymptotic geometric Weyl chambers, we can find a representative with origin $o=[K]\in X$. Writing $M$ for the stabilizer in $K$ of an element of $\cF$, it follows from the Iwasawa decomposition (cf \cite{Helgason}, Chap. VI.) that
$$\cF = G/P \simeq K/M.$$
In particular, the Furstenberg boundary $\mathcal F$ is compact.

Let $w$ be a geometric Weyl chamber and denote by $w(+\infty)$ its equivalence class. We extend as follows the notion of shadow to this setting:
\begin{definition}\label{def-shadow}
For all $x,y\in X$ and all $r>0$, the  shadow of $B(y,r)$ seen from $x$ is defined by 
$$S(x,y,r) = \{ \eta \in \cF \, |\,  \exists w : \fa^+ \tv X, w(+\infty) = \eta,\,  w(0)=x, \,  w(\fa^+) \cap B(y,r) \neq \emptyset \}.$$
\end{definition}
The definition is similar as the strictly negatively curved case, where one just replaces geodesic rays by geometric Weyl chambers.

Let $\xi \in \cF$ and $G_\xi$ the stabilizer of $\xi$ in $G$. There is a unique open dense orbit of $G_\xi$ in $\cF$. It is denoted by $V(\xi)$. These are the points which are \emph{visible} from $\xi$. It is equal to the set of points $\eta\in \cF$ such that there exists a flat $ f : \fa \tv X$, with $f(u)=g. \exp(u).K $ for some $g\in G$, that satisfies : $[f_{|\fa^+}]=\eta$ and $[f_{|-\fa^+}]=\xi$. The complementary set of $V(\xi)$ in $\cF$ is denoted by $L(\xi)$. It follows from Section 5 of \cite{quint2002mesures}\footnote{In this reference, $L(\xi)$ is denoted by $\mathcal Q_\theta^-$} that $L(\xi)$ is a Zariski closed subset of $\mathcal F$.

We will need the following result, showing that $\tau \mapsto L(\tau)$ is continuous in the Hausdorff topology. Recall that $\cF$ identified with $K/M$ and we endow $\cF$ with a $K$-invariant Riemannian metric.

  \begin{lemma}\cite[Lemma 6.6]{deykapovich}\label{lem-DK6.6}
For every $\epsilon>0$, there exists $\delta>0$ such that for $\tau \in \cF$ and for all $\tau'\in B(\tau, \delta)$ :
$$B(L(\tau'), \epsilon) \subset B(L(\tau), 2\epsilon).$$
\end{lemma} 
  Where we denoted by $B(X,\epsilon)$ the $\epsilon$ neighborhood of a subset $X\in \cF$.

\subsubsection{Busemann functions}
The notion of Busemann function extends accordingly to the $\fa^+$-valued distance function $a(\cdot, \cdot)$ defined in section \ref{sec-cartan projection}. Let $x, y \in X$, and $\eta \in \cF$ and take $\xi : \R \tv X$ a geodesic ray at bounded distance of $\eta$.  We define the $\fa$-valued Busemann function by: 
$$
\beta_\eta (x,y) := \lim_{t\tv +\infty} a(x,\xi(t)) - a(y, \xi(t)).
$$
The fact that this limit exists and is independent of the choice of $\xi$ can be founded in Section 6 of \cite{quint2002mesures}, see also \cite[5.34]{benoistquint}.

We recall two geometric inequalities concerning the Busemann functions. We refer to \cite{thirion} which shows precisely the statements which we use in the sequel. Aanalogous statement are shown for instance in Section 6 of \cite{quint2002mesures}.


\begin{lemma}\cite[Proposition 8.66]{thirion}\label{lemma-thirion 1}
There exists a constant $c >0$ such that for all $x, y \in X$ and all $\xi \in S(x,y, r)$ : 
$$\| \beta_\xi (x,y) - a(x,y)\| \leq c r. $$
\end{lemma}
  
\begin{lemma}\cite[Proposition 8.69]{thirion}\label{lemma-thirion 2}
There exists a constant $c >0$ such that for all $x, y \in X$ and all $\xi \in \cF$:
$$\| \beta_\xi (x,y) \|\leq c \|a(x,y)\|.$$
\end{lemma}

\subsection{Dynamic of discrete groups}

We now present general properties of the action of a discrete subgroup $\G<G$ on $X$. The first paragraph presents the notions of limit cone of $\G$ in the Weyl chamber $\overline{\fa^+}$ and limit set of $\Gamma$ in the Furstenberg boundary $\mathcal F$. We detail in the second paragraph different notions of growths and critical exponents. We  conclude  this section by presenting conformal densities in higher rank, with a perspective towards the so-called shadow lemma.  \\

We fix  for the whole section a semisimple, real linear, connected Lie group $G$ with finite center, a choice of Cartan subspace $\fa$ and Weyl chamber $\fa^+$, and a discrete  Zariski dense subgroup $\G\subset G$.

 \subsubsection{Limit set and limit cone} 

The study of how the orbits of the discrete group $\G$ ``accumulate at infinity'' can be splitted in two parts: its accumulation directions the Cartan subspace of $\fp$, which form its \emph{limit cone}, and the accumulation points of the geometric Weyl chamber on the Furstenberg boundary which form its \emph{limits set}.

\begin{definition}
Let $\xi\in \cF$ be fixed. We call \emph{limit set} of $\Gamma$ the set $\Lambda_\Gamma$ of accumulation points of $\Gamma\cdot \xi\in \cF$.

\end{definition}

The following result is a crucial chararcterization of the limit set.

\begin{theorem}\cite{benoist1997proprietes}\label{th-limit set is zariski dense}
If $\Gamma$ is Zariski dense, the limit set $\Lambda_\G$ is a Zariski dense subset of $\cF=G/P$. It is the unique minimal subset of $\cF$ for the action of $\G$. In particular is is independent of the orbit base point.
\end{theorem}

For $u\in \fa$, we denote by $\R^+u \subset \fa$ the half-line generated by $u$. The following set describes the asymptotic directions in $\fa^+$ where $\Gamma$ grows.

\begin{definition}\label{def-limit set}
The \emph{ limit cone} is defined by:
$$C(\G):=\overline{\bigcap_{n\geq 0}  \bigcup_{\g\in \G, \, \| \kappa(\g)\|\geq n} \R^+\kappa(\g)}\subset \overline{\fa^+}.$$
\end{definition}

This limit cone was introduced by Benoist in \cite{benoist1997proprietes}, where the following striking fact was shown.

\begin{theorem}[Benoist, \cite{benoist1997proprietes}]
If $\G$ is Zariski dense, the limit cone $C(\G)$ is a closed convex subset of $\fa$ with non-empty interior.
\end{theorem}

\subsubsection{Critical exponents}\label{sec-critical exponents}

We are interested in understanding the growth of the orbits $\G\cdot o $ inside $X$. We saw in Section \ref{sec-cartan projection} that in higher rank it is interesting to look at an $\fa^+$-valued ``distance'' function taking the direction in the flats into account. In order to study the growth of $\G\cdot o$ inside $X$ in different flat directions, we use linear forms on $\fa$ which are positive on the limit cone. 

Recall that we write $C(\G)\subset \fa$ for  the limit cone of $\G$. Let $\fa^* = \mathcal L(\fa, \R)$ be the set of linear forms on $\fa$ and denote by $C(\G)_+^*\subset \fa^*$ be the set of linear forms  which are positive on the limit cone: 
$$
C(\G)_+^*:=\{\varphi : \fa \tv \R\, |\, \forall x \in C(\G) \backslash\{0\},\, \phi(x) >0 \}. 
$$

 \begin{definition}\label{def-linear form postiive on C(G) and directional critical exponent}
 For all $\varphi \in \fa^*$ we define the critical exponent according to $\varphi$  by :
 $$\delta_\varphi(\G) := \limsup_{R\tv \infty} \frac{1}{R} \log \Card \{ \g \in \G\, |\, \varphi (a(\g\cdot o, o ) ) \leq R \}.$$
 \end{definition}

For every $\varphi \in \C(\G)_+^*$, the critical exponent $\delta_\varphi(\G)$ is independent of the chosen origin (cf Lemma \ref{lemma-triangle inequality in higher rank}) and is finite. Indeed, since $\varphi$ is positive on $ C(\G)$ and linear, there exists $c>0$ such that for all $x\in C(\G)$: 
$$ \varphi(x) > c \| x\|. $$
Therefore, 
$$
\delta_\varphi(\G) \leq  \limsup_{R\tv \infty} \frac{1}{R} \log \Card \{ \g \in \G\, |\,d(g o, o )\leq R/c \},
$$
which is finite since $\G$ is discrete. 

Conversely, if the kernel of $\phi\in \fa^*$ intersects the interior of $C(\G)$, the critical exponent $\delta_\phi(\G)$ is infinite. {We will not study in this paper the case of linear form $\phi\in \fa^*$ which are non-negative on $C(\Gamma)$ but vanishes on $\partial C(\Gamma)$.}


\subsubsection{$\varphi$-Conformal densities}\label{sec-PS measures}

\begin{definition}
For all $\delta>0$, a $\varphi$-conformal density of dimension $\delta$ is a family $\mu = (\mu_x)_{x\in X}$ of mutually absolutely continuous locally finite, borelian measures on $\cF$, such that for all $x,y\in X$,
$$\frac{d \mu_x}{d\mu_y} (\eta) = e^{-\delta \varphi(\beta_\eta(x , y ))}.$$
Such conformal density $\mu$ is $\G$-equivariant if, for all $\g\in \G$ and all $x\in X$, 
$$\g_*\mu_x  = \mu_{\g x}.$$
We denote by $M_\varphi(\delta, \G)$ the space of $\G$-equivariant, $\varphi$-conformal measure of dimension $\delta$. 
\end{definition}

Note that $M_\varphi(\delta, \G)$ is a cone: for all $\mu \in M_\varphi(\delta, \G)$ and all $\lambda>0$, we obviously have $\lambda\mu \in M_\varphi(\delta, \G)$. The following result is the starting point of our approach.

\begin{theorem}[\cite{quint2002mesures}, Section 8]\label{th:quint-patterson-sullivan}
Let $\G<G$ be a Zariski dense subgroup.  Then for all $\varphi \in C(\G)_+^*$ there exists a $\G$ equivariant, $\varphi$-conformal density of dimension $\delta_\varphi(\G)$, which is supported on the limit set $\Lambda_\Gamma \subset \mathcal F$.
\end{theorem}
These densities {have been} constructed in \cite{quint2002mesures} via an adaptation of the usual Patterson-Sullivan construction to higher rank symmetric space. {A similar construction also appears in \cite{Lin04}.} The work of Dey-Kapovich  \cite{deykapovich}  proves their uniqueness (up to scaling) in the case of Anosov subgroup.

In our study of limit cones with respect to taking normal subgroup we will use the following observation, whose analogue in negative curvature is due to Roblin in \cite{roblin2005fatou}.

\begin{lemma}\label{lem-conjugate are density}
Let $\delta>0$ and $\mu\in M_\phi(\delta, \G)$. Let $g\in G$ be an isometry which normalizes $\G$. Then the family $\displaystyle \mu^g = (\mu^g_x)_{x\in X}$ defined for all $x\in X$ by
$\displaystyle \mu^g_x :=  \frac{1}{||\mu_{go}||} g_*^{-1} \mu_{g x}$ is also in $M_\phi(\delta, \G)$ and satisfies $||\mu_o^g|| = 1$. 
\end{lemma}
\begin{proof}

By construction, $||\mu_o^g|| = 1$. For all $z, z' \in X$, $\eta \in \cF$, we have

$$
 \frac{d\mu^g_z}{d\mu^g_{z'}} (\eta) = \frac{d\mu_{gz}}{d\mu_{gz'}}(g\eta)  = e^{-\delta \varphi (\beta_{g\eta} (gz,gz'))}= e^{-\delta  \varphi (\beta_{\eta} (z,z'))}. 
$$

Therefore the family $\nu^g$ is a $\varphi$-conformal density of dimension $\delta$.

\medskip

Moreover, for all $\g \in \G$, we have

$$
\g_* \mu_z^g = \frac{1}{\| \mu_{go} \| }\g_* g_*^{-1} \mu_{g z}=  \frac{1}{\| \mu_{go} \| } g_*^{-1} \left( g \g g^{-1}\right)_* \mu_{g z}=  g_*^{-1}\mu_{(g\g g^{-1}) g z}= \mu_{\g z } ^{g}.
$$

Therefore, $\mu^g\in M_\phi(\delta, \Gamma)$.

\end{proof}

The crucial feature linking the measurable properties  of the $\varphi$-conformal densities with  the dynamical action  of $\G$ is the so-called \emph{Shadow lemma}, which is originally due to Sullivan in \cite{Sul79} for hyperbolic manifolds. The same proof extends to CAT$(-1)$ metric spaces, see Lemma 1.3 of \cite{roblinergodicite}. In the higher rank setting, several variant of the Shadow Lemma are known, depending on the boundary and the densities which are considered. We will use the following version, which first appeared as Lemme 8.2 of \cite{quint2002mesures} {and Theorem 4.7 of \cite{Lin04}}. 

\begin{theorem}[Shadow Lemma, \cite{quint2002mesures,deykapovich} ]\label{lem-shadow lemma}
Let $\mu$ be a $\G$-equivariant, $\varphi$-conformal density of dimension $\delta>0$. For all $x\in X$ there exists $R_0>0$ and $C>0$ such that for all $r>R_0$, and for all $\g\in \G$. 
$$\frac{1}{C} e^{-\delta_\varphi \varphi(a(x, \g x))} \leq \mu_x( S(x,\g x, r ) )\leq C  e^{-\delta_\varphi \varphi(a(x, \g x))}.$$
\end{theorem}

We will now briefly present a proof of this Shadow Lemma, {since we will need most of its ingredients in the sequel}. We follow the same strategy as in Lemma 8.2 of \cite{quint2002mesures} eventhough the notations have evolved in the past 20 years. We {rely on two key lemmas which we quote from \cite{deykapovich}}. Equivalent statements are shown in Section 5 of \cite{quint2002mesures}.

We endow $\cF \simeq K/M$ with a $K$-invariant Riemmanian metric.  We have seen in \ref{def-shadow} the definitions of shadows $S(x,y,r)$, visible sets $V(\xi)$  and its complement $L(\xi)$. Our terminology and notations are close to \cite{deykapovich}, where we can find the two following lemmas. Recall that we have fixed a $K$-invariant Riemannian metric on $\mathcal F$ and that $\G$ is supposed to be Zariski dense in $G$. The point $o = [K]\in X = G/K$ provides us a fixed origin.

\begin{lemma}\cite[Lemma 6.7]{deykapovich}\label{lem-DK 6.7}
{For all $x\in X$}, there exists $q<\mu_x(\Lambda_\G)$ and $\epsilon$ such that for all $\tau \in \Lambda_\G$, and all Borel subsets $B$ of $\cF$  contained in the $\epsilon$ neighborhood of $L(\tau$), one has : $\mu_x(B) \leq q$. 
\end{lemma}

\begin{lemma}\cite[Lemma 6.8]{deykapovich}\label{lem-DK 6.8}
For every $\epsilon>0$, there exists $r_1>0$ such that for all $r\geq r_1$ and all $x\in X$, the complement of $S(x, o,r)$ in $\cF$ is contained in a $\epsilon$-neighborhood of $L(\tau)$ for some $\tau$ in $S(o, x,0)$.
\end{lemma}

\begin{proof}[Proof of Theorem \ref{lem-shadow lemma}]

{We fix $x\in X$. Let us remark first} that {since $\mu$ is $\delta$-conformal}, for all $y\in X $ and all $r>0$ {we have} 
\begin{eqnarray}\label{eq:conformal-invariance}
\mu_x(S(x,y , r) )= \int_{S(  x, y , r)} e^{-\delta \varphi (\beta_{\xi} ( x, y)) } d\mu_{y}	(\xi)	
\end{eqnarray}
{By} Lemma \ref{lemma-thirion 1}, there exists $c>0$ such that for all $\xi \in S(  x, y , r)$, one has $|\varphi (\beta_{\xi} ( x, y)) - \varphi(a(x,y))| \leq cr$. This implies that
\begin{eqnarray}\label{shadow fundamental inequality}
\mu_x(S(x,y , r) )\leq e^{\delta c r} \mu_y(S(x,y,r))  e^{-\delta \varphi(a (x,y))} \leq C||\mu_y|| e^{-\delta \varphi(a (x,y))}.
\end{eqnarray}
Moreover since $\mu$ is $\G$-invariant, we have for all $\gamma\in \G$ $||\mu_{\gamma x}|| = ||\mu_x||$. This implies the upper bound given by the Shadow Lemma.

Let us show the lower bound.
\begin{eqnarray*}
\mu_x(S(x,\g x , r) ) &=& \mu_{\g^{-1} x} (S( \g^{-1} x, x , r)) \\
 							&=& \int_{S( \g^{-1} x, x , r)} e^{-\delta \varphi (\beta_{\xi} (\g^{-1} x, x)) } d\mu_{x}\\
 							& \geq  & \frac{1}{e^{cr}}\mu_x (S( \g^{-1} x, x , r))  e^{-\delta \phi(a(x,\g x )) },
\end{eqnarray*}
the last inequality coming from Lemma \ref{lemma-thirion 1}. Now using  Lemmas \ref{lem-DK 6.7} and \ref{lem-DK 6.8}, we get the lower bound $\mu_x (S( \g^{-1} x, x , r)) \geq  \frac{1}{C}$, which finishes the proof.

\end{proof}
Note that the Shadow Lemma implies in particular that if $\mu\in M_\phi(\delta, \Gamma)$, then there exists $C>0$ such that for all $\gamma\in \Gamma$,
\begin{equation}\label{eq:lower-bound-mass}
\|\mu_{\gamma o}\|\geq \frac 1 C e^{-\delta \phi(a(o, \gamma o))}.
\end{equation}

{This standard version of the Shadow Lemma is only valid on orbits of the discrete group $\Gamma$. The proofs of our main result rely on an extension of this Shadow Lemma, called \emph{Shadow principle}, which holds on much larger sets: the orbits of the full normalizer of $\G$ in $G$. This was first proven in negative curvature by Roblin in \cite[Théorème 1.1.1]{roblin2005fatou}. We will extend it to higher rank in the next section.}

\section{Normal subgroup and asymptotic invariants}\label{sec-subgroup and critical exponent}

In this section we prove the first main theorems of the paper. In Section \ref{sec-extending the shadow lemma}, we prove our main tool, given in Theorem \ref{th-lem-shadow_principle_normalizer}, which is the extension of the Shadow Lemma to the orbits of the normalizer of any Zariski dense discrete group. As a corollary we get {that the limit cone of a Zariski dense group is the same as the limit cone of any of its Zariski dense normal subgroups}. We eventually prove in Section \ref{sec- normal subgroup and critical exponent} the inequality between critical exponents announced in Theorem \ref{th-main limit cone} .

\subsection{Shadow principle and limit cones}\label{sec-extending the shadow lemma}

The following extension of the Shadow Lemma was called \emph{shadow principle} by Roblin in \cite{roblin2005fatou}.

\begin{definition}
Let {$\Gamma<G$ be a Zariski dense discrete subgroup,} $\phi\in C(\Gamma)_+^*$ and $\delta>0$. We say that a set $Y\subset X$ satisfies the \emph{shadow  principle} {for $M_\phi(\delta, \Gamma)$.} if there exists $R,C>0$ such that for all $\mu\in M_\phi(\delta, \Gamma)$, all $r\geq R$ and all $x,y \in Y$,
$$\frac{1}{C} \| \mu_y \| e^{-\delta \varphi(a(x,y))} \leq \mu_x( S(x,y, r ) )\leq C\| \mu_y \|  e^{-\delta (\G) \varphi(a(x, y))}.$$
\end{definition}

It follows from (\ref{shadow fundamental inequality}) that the upper bound in the shadow principle is a mere consequence of the $\phi$-conformality of dimension $\delta$, valid for all $Y\subset X$. All the interest of this property relies hence in the lower bound. Note also that the shadow principle is invariant by scaling of the conformal density. It is hence enough to check it on conformal densities such that $||\mu_o|| = 1$, where $o\in X$ is a fixed origin. The end of this paragraph is devoted to the proof of the following result, which extends Theorem 1.1.1 of \cite{roblin2005fatou} to higher rank symmetric spaces.

\begin{theorem}\label{th-lem-shadow_principle_normalizer}
Let $G$ be a real linear,  semisimple, connected, Lie group with finite center and $\G$ be a Zariski dense subgroup of $G$. Let $N(\G)$ be the normalizer of $\G$ in $G$.

For all $\phi\in C(\Gamma)_+^*$, all $\delta>0$ and all $x\in X$, the orbit $Y = N(\G)\cdot x$ satisfies the shadow principle for $M_\phi(\delta, \Gamma)$.\footnote{The result is true for all $\delta>0$. However, {it follows from \cite{quint2002mesures} that} when $\delta<\delta_\phi(\G)$ then $M_\phi(\delta, \Gamma)=\emptyset$. }
\end{theorem}

Let us first reduce the proof to a simpler statement, where one of the base points is fixed.
\begin{lemma}\label{lem - transform left inequality}
Let $G$ be a real linear,  semisimple, connected, Lie group with finite center and $\G$ be a Zariski dense subgroup of $G$, with normalizer $N(\G)$. Let us fix $\phi\in C(\Gamma)_+^*$, $\delta>0$ and $o\in X$. 

The orbit $Y = N(\G)\cdot o$ satisfies the shadow principle for $M_\phi(\delta, \Gamma)$ if and only if there exists $R,C>0$ such that for all $y\in Y$ and all $\mu\in M_\phi(\delta, \Gamma)$,
\begin{equation}\label{eq:weak-shadow-principle}
\mu_y(S(o,y,R)) \geq \frac 1 C ||\mu_y||.
\end{equation}
\end{lemma}

\begin{proof}
{We fix $o\in X$ and denote by $Y=N(\G)\cdot o$ the orbit of $o$ under the normalizer of $\G$ in $G$. As already mentionned, we only have to deal with the lower bound of the shadow principle, since the upper bound is satisfied for all $Y\subset X$. Note also that, by Equation (\ref{eq:conformal-invariance}) and Lemma \ref{lemma-thirion 1}, if $Y$ satisfies the shadow principle then (\ref{eq:weak-shadow-principle}) is satisfied for all $y\in Y$. Therefore, we are only left with showing that, if for all $\mu\in M_\phi(\delta, \Gamma)$ and all $y\in Y$, the inequality (\ref{eq:weak-shadow-principle}) is satisfied, then $Y$ satisfies the shadow principle.

\medskip

Assume therefore that (\ref{eq:weak-shadow-principle}) is satisfied for all $\nu\in M_\phi(\delta, \Gamma)$ and all $y\in Y$. Let $\mu\in M_\phi(\delta, \Gamma)$ be fixed. Let $x,y\in Y$, and $g\in N$ such that $g o =x$. We have seen in Lemma \ref{lem-conjugate are density} that the family $\mu^g$ defined for all $z\in X$ by $\displaystyle \mu^g_z:= \frac{1}{\| \mu_{go} \| } g_*^{-1} \mu_{g z}$ is also in $M_\phi(\delta, \Gamma)$. Moreover, setting $z := g^{-1} y$, we have: 
\begin{eqnarray*}
\mu^g_z (S (o, z, R) ) &= & \frac{1}{\| \mu_{g o} \| } \mu_{g z } (g S (o,z,R)) \\
							&=& \frac{1}{\| \mu_x \| } \mu_y (S (x,y,R)). 
\end{eqnarray*}
Therefore, since by (\ref{eq:weak-shadow-principle}) we have}
$$
\mu^g_z (S (o, z, R) ) \geq \frac{1}{C} \| \mu^g_z\|= \frac{1}{C}\frac{\| \mu_y \| }{\|\mu_x\|},
$$
we get
$$\mu_y (S(x,y,R)) \geq \frac{1}{C}\| \mu_y\|,$$
which is equivalent to the lower bound in the shadow principle by Equation (\ref{eq:conformal-invariance}) and Lemma \ref{lemma-thirion 1}.

\end{proof}

\begin{proof}[Proof of Theorem \ref{th-lem-shadow_principle_normalizer}]
By Lemma \ref{lem - transform left inequality}, it is enough to show that there exists $C,R>0$ such that for all $\mu\in M_\phi(\delta, \Gamma)$ with $||\mu_o|| = 1$ and all $y\in Y = N(\G)\cdot o$, the lower bound (\ref{eq:weak-shadow-principle}) is satisfied. We follow the strategy developped in \cite{roblin2005fatou}.

Suppose by contradiction that it is not the case. There exists a sequence $(R_i)_{i\in \mathbb N}$ with $R_i \tv \infty$, a sequence $\displaystyle (g_i) \in N(\Gamma)^{\mathbb N}$ and a sequence $\displaystyle (\mu^i)_{i\in \mathbb N} \in M_\varphi(\delta, \G)^{\mathbb N}$ with $||\mu^i_o|| = 1$ such that 
\begin{equation}\label{eq-contradiction}
\lim_{i\to +\infty} \frac{1}{\|\mu_{g_i o} ^i \|} \mu_{g_i o} ^i (S(o,g_i o , R_i)) = 0.
\end{equation}
Denote by $\nu^i = (\nu^i_x)_{x\in X}$ the family of measures defined for all $x\in X$ by $\displaystyle \nu_x^i := \frac{1}{\| \mu^i_{g_io} \| } {g_i}_*^{-1} \mu^i_{g_i x}$. By Lemma \ref{lem-conjugate are density}, it is a {$\phi$-conformal density of dimension $\delta$} and satisfies $||\nu^i_o|| = 1$.

Since $\cF$  is compact and $(\nu_o^i)_{i\in \mathbb N}$ is a sequence of probabilities on $\cF$, we can assume up to taking a subsequence that it converges weakly to some probability measure $\nu^{\infty}_o$. Then since $\nu^i$ is $\phi$-conformal of dimension $\delta$, for all $x\in X$ the sequence of finite measures $(\nu_x^i)_{i\in \mathbb N}$ also converges in the weak topology to some finite measure $\nu^{\infty}_x$ and the family $\nu^\infty = (\nu^{\infty}_x)_{x\in X}$ is in $M_\varphi(\delta, \G)$.

\begin{lemma}\label{Lemme interne}
Up to extracting a subsequence, there exists $\tau  \in \cF $ such that for all $\epsilon>0$ there exists $i_0>0$ such that  for all $i\geq i_0$: 
$$\cF \setminus S(g_i^{-1} o , o , R_i) \subset  B(L(\tau), \epsilon).$$
\end{lemma}
\begin{proof}
We first apply Lemma \ref{lem-DK 6.8} :  for all $\epsilon >0$ there exists $r>0$ such that for all $R>r$ for all $y \in X$ :
$$ \cF\setminus S(y,o ,R) \subset B(L(\tau_y) , \epsilon),$$
for some $\tau_y \in S( o , \{y \})$. Taking $y=g_i^{-1} o $, gives a sequence $\tau_i\in \cF$ such that $ \cF\setminus S( g_i^{-1} o , o, R) \subset B(L(\tau_i) , \epsilon).$
 By compactness, we can suppose that $\tau_i$ converges to $\tau \in \cF$. 

We now apply Lemma \ref{lem-DK6.6} : for $i\gg 1$ sufficiently large, 
$$B(L(\tau_i) , \epsilon) \subset B(L(\tau) , 2\epsilon).$$

Therefore, for $i\gg 1$ sufficiently large, we get $R_i>r$ and 
$$ \cF\setminus S(o,g_i^{-1} o ,R_i) \subset B(L(\tau) ,2 \epsilon),$$
which concludes the proof of Lemma \ref{Lemme interne}.
\end{proof}

Let $V$ be an open subset of $\cF$ such that $\overline{V} \cap L(\tau) = \emptyset$. By Lemma \ref{Lemme interne}, there exists $i_0$ such that for all $i\geq i_0$ we have
\begin{eqnarray*}
\nu_o^{\infty}(V) & \leq & \liminf_{i\tv \infty} \nu_o^i  (S(g_i^{-1} o , o , R_i) ) \\
						& \leq &  \liminf_{i\tv \infty}\frac{1}{\| \mu_o^i\|} \mu_{g_i o} ^i (S(o,g_i o , R_i))		\end{eqnarray*}

Using equation (\ref{eq-contradiction}), we get $\nu_o^\infty(V)=0$. Hence $\nu_o^\infty$ is supported on $L(\tau)$. In particular, by $\Gamma$-invariance of the family $\nu^\infty$, the set $L(\tau)$ contains a closed set which is $\G$ invariant. Therefore $\Lambda_\Gamma\subset L(\tau)$ which is absurd by Theorem \ref{th-limit set is zariski dense} since $\G$ is Zariski dense {and $L(\tau)$ is Zariski closed}. 

\end{proof}

We now prove the equality of the limit cone for Zariski dense normal subgroups.

\begin{corollary}\label{coro:limit-cones}
Let $\G$ be a {discrete} subgroup of $G$ and $\G' \triangleleft \Gamma$ be  a normal subgroup of $\G$. Suppose $\G'$ is Zariski dense,  then the limit cone of $\G'$ coincides with the limit cone of $\G$.
\end{corollary}

\begin{proof}
If it were not the case, since both limit cones are closed, it means that there exists an open cone  $U\subset C(\Gamma) \backslash C(\Gamma')$. This implies that there is  a linear form $\varphi$ in $C(\G')_+^*$ whose kernel intersects $U$.

We have seen in Section \ref{sec-PS measures} that since $\Gamma'$ is Zariski dense, there exists $\mu = (\mu_x)_{x\in X}\in M_\phi(\delta_\phi(\Gamma'), \Gamma')$. Since we supposed $\ker(\varphi)$ intersects {$U$ which is contained in the interior of $C(\Gamma)$,} there exists a sequence of elements $\g_n$ of $\G$ such that $\varphi(a(x, \g_n x)) \tv -\infty$. {Nevertheless, by Theorem \ref{th-lem-shadow_principle_normalizer}, the orbit $\G\cdot x$ satisfies the shadow principle for $M_\phi(\delta, \Gamma')$. Therefore we also have

$$\mu_x(\Lambda_{\G'}) \geq \mu_x( S(x, \g_n x, r ) ) \geq \frac{1}{C}e^{-\delta_\phi(\G') \varphi (a(x, \g_n x))}.$$
Since the right hand side goes to $+\infty$ with $n$, we get a contradiction.}
\end{proof}

Note that, in the proof, we use the fact that $\Gamma'$ is Zariski dense to ensure the existence of $\phi$-conformal densities. This prevent from ``counter-examples'' such as $\G'=\G_\ell\times \{\Id\} \subset \G_\ell \times \G_r =\G$ for which, as mentionned in the introduction, Corollary \ref{coro:limit-cones} is trivially false.

We now prove a necessary  and sufficient condition on $\G$ for its normal subgroups to be Zariski dense.

\begin{definition}
Let $G$ be a semisimple Lie group. A subgroup $\G\subset G$ is said \emph{normally irreducible} if for all normal subgroups $N\subset G$, if $N\neq G$ then:
$$N \cap \G=\{ \Id \}.$$
\end{definition}

The proof of the following result has been suggested to us by Y. Benoist.

\begin{proposition}\label{prop-CNS normally irred}
{Let $G$ be a real linear semisimple Lie group. Let $\G\subset G$ be a Zariski dense subgroup. 

All non trivial normal subgroups $\G' \triangleleft \G$ are Zariski dense if and only if 
$\G$ is normally irreducible.}
\end{proposition}

\begin{proof}
Let $H:=\overline{\G'}^Z$ be the Zariski closure of $\G'$. It is an algebraic set, therefore 
the condition $ghg^{-1}\in H$ is algebraic. This condition  is satisfied for all $(h,g) \in \G\times \G'$, hence for all $(h,g)\in H\times G$. In other words, $H$ is normal in $G$. 

Suppose that $H\neq G$, then by irreducibility $H\cap \G =\{\Id\}$ which its absurd since $\G'\subset (H \cap \G)$. 
\medskip

Suppose {now} that $\G$ is not normally irreducible. Decompose $G$ into the Cartesian product of its simple factors $G = \Pi_{i=1}^\ell G_i$ remark that $\ell>1$. Moreover by simplicity of the $G_i$, if $N\neq G$, $N$ is also a Cartesian product and in particular is not Zariski dense. 
By assumption, there exists $N\neq G$  a normal subgroup of $G$, such that:
$N \cap \G \neq \{\Id \}\triangleleft \G.$ This is a normal subgroup of $\G$ which is not Zariski dense.  

\end{proof}

\subsection{Normal subgroup and critical exponent}\label{sec- normal subgroup and critical exponent}

Let $\G$ be a discrete subgroup of $G$ and $\G' \triangleleft \Gamma$ be  a Zariski dense normal subgroup of $\G$. Let $\phi\in C(\Gamma')_+^*$ be fixed. We now finish the proof of Theorem \ref{th-main critical exponent} by showing the lower bound inequality for all the critical exponents $\delta_\varphi(\Gamma)$ and $\delta_\phi(\Gamma')$.

\begin{definition}
A set $ Y \subset X$   is said to have \emph{bounded geometry} if for all $R>0$ there exists $A_R>0$ such that for all $y\in Y$,
$$\Card (Y \cap B(y,R) ) \leq A_R.$$
\end{definition}

In particular any orbit of a discrete group of isometry of $X$ has bounded geometry.

The following proposition appears in Roblin \cite[Lemma 1.2.4] {roblin2005fatou} for conformal densities in CAT($-1$) space.  We extend it to higher symmetric spaces. 

\begin{proposition}\label{pr-Roblin combinatorial}
Let $\delta>0$ and $Y\subset X$ be a set with bounded geometry which satisfies the shadow principle for $M_\phi(\delta, \Gamma)$. 

Then for all $\mu = (\mu_x)_{x\in X}\in M_\phi(\delta, \Gamma)$, the critical exponent of the series
$$
 s \mapsto \sum_{y \in Y} \| \mu_y \| e^{-s \varphi( a (o, y )) }
$$
is at most $\delta$. 
\end{proposition}

\begin{proof}[Proof of Proposition \ref{pr-Roblin combinatorial}]
Let $r>1$ be the real rank of $G$ and $(\epsilon_i)_{i=1,\ldots ,r}$ be a basis of $\fa^*$, such that $\varphi =\sum_i b_i \epsilon_i$ with $b_i>0$ for all $i\in \{ 1, \ldots, r\}.$

For all $(n_1,..., n_r)\in \mathbb N^r$, we define
$$
R_{n_1, \ldots, n_r} = \bigcap_{i = 1}^r\{ y \in X\,  |\,  \epsilon_i(\alpha(o,y)) \in [n_i ,n_i+1)\}.
$$ 
Our proof rely on the following uniform multiplicity bound for the covering of $Y$ by the shadows $\{ S(o,y,r_0)\; ; \; y\in R_{n_1,\ldots, n_r}\} $.

\begin{lemma}\label{lem-bounded}
Let $Y$ be a set with bounded geometry.  For all $R_0>0$, there exists $B_{R_0}>0$ such that for all $(n_1,..., n_r)\in \mathbb N^r$, the covering $\displaystyle \bigcup_{y\in {Y\cap} R_{n_1,\ldots, n_r}} S(o,y,R_0)$ has multiplicity  bounded by $B_{R_0}$.
\end{lemma}
\begin{proof}
{Let us fix $r>0$,  $(n_1,..., n_r)\in \mathbb N^r$. For all $y_1\in Y$, we write 
$$
{\rm cov(y_1)} = \#\{ y_2 \in R_{n_1,\ldots, n_r}\; ; \; S(o, y_1, R_0) \cap S(o, y_2, R_0)\neq \emptyset\}.
$$
We have to show that ${\rm cov(y_1)}$ is bounded by a constant depending only on $R_0$. Let us fix $y_1\in Y$, and let $y_2\in Y$ and $\eta \in \cF$ be such that}
$$\eta \in S(o, y_1, R_0) \cap S(o, y_1, R_0).$$
We represent $\eta$ by a geometric Weyl chamber $w  : \fa_+ \tv X$.  By definition of shadows, there exists $u_1, u_2 \in \fa_+$ such that 
\begin{equation}\label{eq-d(x1,y1)<r et d'x2,y2) <r}
\left\{
\begin{array}{c}
w(u_1) \in B(y_1, R_0)\\
w(u_2) \in B(y_2, R_0)
\end{array}
\right.
\end{equation}

Therefore, using Lemma \ref{lemma-triangle inequality in higher rank}, there exists $C_1>0$ depending only on $G$ such that: 
\begin{equation}\label{eq-  ld(x1,x2) - d(y1,y2)l <r }
\| a(w(u_1),w(u_2))  -a(y_1,y_2) \| \leq C_1R_0.
\end{equation}

Equation (\ref{eq-d(x1,y1)<r et d'x2,y2) <r}) also implies that there exists $C_2>0$ depending on $G$ such that: 
\begin{equation}\label{eq-d(x,o)-d(y,o)< r }
\left\{
\begin{array}{c}
\|a(w(u_1),o) -a(y_1,o) \| \leq C_2 R_0\\
\|a(w(u_2),o) -a(y_2,o) \| \leq C_2 R_0\\
\end{array}
\right.
\end{equation}
Since $u_1, u_2\in \fa$ the vectors $a(w(u_1),o) -a(w(u_2), o )$ and $a(w(u_1),w(u_2))$ coincide up to the action of an element of the Weyl group. To simplify notations, we will suppose that $a(w(u_1),o) -a(w(u_2), o ) = a(w(u_1),w(u_2))$, adaptation to the general case is immediate {using Lemma \ref{lemma-triangle inequality in higher rank}}. We have hence
$$ \| a(w(u_1), w(u_2))  - a(y_1,o)+ a(y_2,o) \| \leq 2C_2 R_0.$$
{Moreover,} since $y_1, y_2\in R_{n_1, \ldots n_2}$, we get for all $i\in \{ 1, \ldots, r\}$:
$$\epsilon_i \left(a(w(u_1), w(u_2))\right)   \leq 2C_2 R_0 + 2.$$
Since $(\epsilon_i)_{i\in \{ 1, \ldots, r\}}$  forms a basis, there exists hence $C_3>0$ depending only on $R_0$ and $G$, such that:
$$\|a(w(u_1), w(u_2)) \| \leq C_3.$$
Therefore, using Equation (\ref{eq-  ld(x1,x2) - d(y1,y2)l <r }) we get
$$\| a(y_1,y_2) \| \leq \|a(w(u_1), w(u_2)) \|  +C_1 R_0\leq C_4$$
for a constant $C_4>0$ depending only on $r_0$ and $G$. {Since $Y$ has bounded geometry, this implies that 
$$
{\rm cov(y_1)} \leq A_{C_4},
$$
where $A_{C_4}$ is the maximum number of elements of $Y$ contained in a ball of radius $C_4$ and only depends on $R_0$.}
\end{proof}

{We now resume the proof of Proposition \ref{pr-Roblin combinatorial}. Let us fix $\mu\in M_\phi(\delta, \G)$. Since $Y$ satisfies the shadow principle for $M_\phi(\delta, \Gamma)$, there exists $C, R_0$ such that for all $y\in Y$, we have
$\displaystyle \mu_y(S(o,y,R_0)) \geq C ||\mu_y||$. Moreover, it follows from Lemma \ref{lem-bounded} that there exists $B_{R_0}>0$ such that}  the covering  $\bigcup_{y\in R_{n_1, \ldots, n_r} } S(o,y,r_0)$ has multiplicity  bounded by $B_{R_0}$. 

{This implies that for all $(n_1,..., n_r)\in \mathbb N^r$, there exists $K>0$ such that}
\begin{eqnarray*}
\|\mu_o \| &\geq & \frac{1}{{B_{R_0}}} \sum_{y \in R_{n_1, \ldots, n_r}} \mu_o( S(o,y, r_0)) \\
				&\geq & K' e^{-\delta \sum_i (b_i n_i)} \sum_{y \in R_{n_1, \ldots, n_r}} \|\mu_y\|. \\ 
\end{eqnarray*}
{We get therefore }
\begin{eqnarray*}
 \sum_{y \in Y} \| \mu_y \| e^{-s \varphi( a (o, y )) } &=& \sum_{n_1,\ldots n_r} \sum_{y \in R_{n_1,\ldots n_r}  } \| \mu_y \| e^{-s \varphi(a(o,y))} \\
 																	& \leq & K'^{-1} \|\mu_o\| \sum_{n_1,\ldots n_r} e^{-s \sum_i (b_i n_i) } e^{\delta \sum_i (b_i n_i) }.
\end{eqnarray*}

Eventually, the last series converges for all $s>\delta$.
\end{proof}

A {direct} corollary of Proposition \ref{pr-Roblin combinatorial} is the following result,  originally due to Roblin \cite[Theorem 2.2.1]{roblin2005fatou} for groups acting on CAT($-1$) spaces.

\begin{corollary}\label{cor-1/2 delta0 leq delta1, higher rank}
Let $\G$ be a {discrete} subgroup of $G$ and $\G'$  a Zariski dense normal subgroup of $\G$. For all $\phi\in C(\Gamma')_+^*$ the critical exponents {according to $\phi$} satisfy:
$$\frac{1}{2}\delta_\varphi(\G) \leq \delta_\varphi(\G').$$
\end{corollary}

\begin{proof}
Since $\G'$ is Zariski dense, {by Theorem \ref{th:quint-patterson-sullivan}} there exists $\mu\in M_\phi(\delta_\phi(\Gamma'), \G')$. By Theorem \ref{th-lem-shadow_principle_normalizer}, we know that $\G\cdot o $ satisfies the shadow principle for $M_\phi(\delta_\phi(\Gamma'), \G')$. Therefore, for all $s>\delta_\phi(\G')$ Proposition \ref{pr-Roblin combinatorial} implies
$$  \sum_{\g \in \G} \| \mu_{\g o } \| e^{-s \varphi (a( o, \g o )) } <  +\infty.$$

Moreover by (\ref{eq:lower-bound-mass}), there exists $C>0$ such that for all $\gamma \in \G$, we have 
$\displaystyle  \|\mu_{\gamma o}\|\geq \frac{||\mu_o||}{C} e^{-\delta_\phi(\Gamma') \phi(a(o, \gamma o))}$.
We get hence that for all $s>\delta_\varphi(\G')$,
\begin{eqnarray*}
\sum_{\g \in \G} e^{- (s+\delta_\varphi(\G') )  \varphi (a( o, \g o ))}  \leq \frac C {||\mu_o||}\sum_{\g \in \G} \| \mu_{\g o } \| e^{-s \varphi (a( o, \g o )) } <  +\infty.
\end{eqnarray*}
Therefore $2\delta_\phi(\Gamma') \geq \delta_\phi(\Gamma)$.

\end{proof}

\section{Amenability and critical exponent}\label{sec -amenability and critical exponent}

\subsection{Amenable discrete groups}

Given a discrete set $Y$, we write $\ell^\infty(Y)$ for the set of \emph{bounded maps} from $Y$ to $\mathbb R$.

\begin{definition}
A discrete group $\mathcal G$ is \emph{amenable} if there exists a positive linear map $\Theta : \ell^\infty (\cG) \to \mathbb R$, such that $\Theta({\bf 1}_{\mathcal G}) = 1$ and which is right invariant, i.e. for all $g\in \mathcal G$ and all $f \in \ell^\infty(\cG)$, we have
$$
\Theta(a \mapsto f(ag)) = \Theta(f).
$$

Such map $\Theta$ is called a \emph{right-invariant mean} on $\mathcal G$.

\medskip
If $\Gamma$ is a discrete group, a normal subgroup $\Gamma'\triangleleft\Gamma$ is \emph{co-amenable} in $\Gamma$ if the quotient group $\Gamma'\backslash \Gamma$ is amenable.
\end{definition}

We refer for instance the reader \cite{Jus15} for various characterizations and examples of amenable discrete groups. Let us present an important characterization of amenability, which illustrates how this property is related to the growth of the group.

\medskip

Let $\mathcal G$ be a discrete group. For all \emph{finite} set $S\subset \mathcal G$, and all subset $A\subset \mathcal G$, we define the $S$\emph{-boundary} of $A$ to be
$$
\partial_S A = \{g\in \mathcal G\backslash A \; ; \; \exists (g',s)\in A\times S, g = s g'\}.
$$
The $S$-\emph{isoperimetric constant} of $\mathcal G$ is then given by
$$
\mathfrak h_S(\mathcal G) = \inf_{A\subset G \mbox{ finite}} \frac {\# \partial_S A}{\# A}.
$$
The following characterization of amenability has been shown by F\"olner in \cite{Fol55}.

\begin{theorem}[F\"olner \cite{Fol55}]
A discrete group $\mathcal G$ is amenable if and only if for all finite $S\subset \mathcal G$, the isoperimetric constant $\mathfrak h_S(\mathcal G)$ is $0$.
\end{theorem}

This easily implies that all finitely generated groups with sub-exponential growth (e.g. infra-nilpotent discrete groups) are amenable, whereas any free group with at least 2 generators is non-amenable. 

\subsection{Averaging conformal density}\label{sec-averaging conformal density}

Let us fix a discrete Zariski dense subgroup $\Gamma$ of $G$, and a Zariski dense normal subgroup $\Gamma' \triangleleft \Gamma$. By the result of last section, we know that the limit cones of $\G$ and $\G'$ coincide. Let $\phi\in C(\G)_+^*$ and $\delta>0$ be fixed. We write 
$$
\|\phi\|  = \sup_{v\in \fa \backslash \{0\}} \frac{|\phi(v)|}{\|v\|} <+\infty.
$$

We assume from now on that $\Gamma'$ is co-amenable in $\Gamma$. Let $\Theta$ be a right-invariant mean on $\Gamma'\backslash \Gamma$. Let $\mu\in M_\phi(\delta, \Gamma')$ with $||\mu_o||= 1$ be fixed. Recall that for all $u\in \Gamma$, we write $\mu^u$ for the conformal density defined for all $x\in X$ by
$$
\mu^u_x = \frac{1}{\|\mu_{uo}\|}(u^{-1})_* \mu_{ux}.
$$
By Lemma \ref{lem-conjugate are density}, we still have $\mu^u\in M_\phi(\delta, \Gamma')$ and $\| \mu_o^u\| = 1$. Therefore for all $x\in X$ and all continuous maps $f : \mathcal F\to \R$, the map
$$
\mathcal L_{x,f} : \left\{\begin{array}{ccc}
						\Gamma & \to & \R\\
						u & \mapsto & \int_{\mathcal F}f(\xi) d\mu_x^u(\xi)
					\end{array}\right.
$$
is $\G'$-invariant. By Lemma \ref{lemma-thirion 1}, there exists $C>0$ such that $\mathcal L_{x,f}$ is bounded from above by $C e^{\delta \|\phi\| d(o,x)} \sup{|f|}$. Therefore, $\mathcal L_{x,f}$ induces a bounded map on $\G'\backslash \G$ and we can define the following application from $\mathcal C(\mathcal F, \R)$ to $\mathbb R$ by
\begin{equation}\label{eq:mean-density}
\mu_x^\Theta : f \mapsto \Theta\left(u \mapsto \mathcal L_{x,f}(u) = \frac{1}{||\mu_{uo}||}\int_{\mathcal F} f(u^{-1}\xi)d\mu_{ux}(\xi)\right).
\end{equation}

We get the following.

\begin{lemma}\label{lem:average-density}
The family $\mu^\Theta = (\mu^\Theta_x)_{x\in X}$ induces a $\Gamma'$-invariant $\phi$-conformal density of dimension $\delta$, still written $\mu^\Theta$, which we call the \emph{$\Theta$-average} of $\mu$. Moreover, it satisfies for all $\gamma\in \Gamma$,
\begin{equation}\label{eq:big-mass}
||\mu^\Theta_{\gamma o}|| \cdot ||\mu^\Theta_{\gamma^{-1}o}||\geq 1.
\end{equation}

\end{lemma}

\begin{proof}
It follows from (\ref{eq:mean-density}) that for all $x\in X$, $\mu^\Theta_x$ is a positive linear map from $\mathcal C(\mathcal F, \R)$ to $\mathbb R$, with $\nu_x({\bf 1}_{\mathcal F}) \leq e^{\delta d(o,x)}$. Therefore it induces a unique finite Borel measure on $\mathcal F$, which we still denote by $\mu^\Theta_x$. Moreover, it immediately follows from Lemma \ref{lem-conjugate are density} that the family $\mu^\Theta = (\mu_x^\Theta)_{x\in X}$ belongs to $M_\phi(\delta, \Gamma)$.

\medskip
Eventually, let us fix $\gamma\in \Gamma$. Since $\Theta$ is invariant under right multiplication by elements of $\Gamma$, we have :
\begin{eqnarray*}
\|\mu^\Theta_{\gamma o } \| &=& \Theta \left(  u \mapsto \frac{\| \mu_{u \gamma o }\| }{\| \mu_{u o }\|} \right)\\
					&=& \Theta \left(  u \mapsto \frac{\| \mu_{u o }\| }{\| \mu_{u \gamma^{-1} o }\|} \right)\\
					&\geq & \left( \Theta\left(  u \mapsto \frac{\| \mu_{u \gamma^{-1} o }\| }{\| \mu_{u  o }\|} \right) \right)^{-1} = \| \mu^\Theta_{\gamma o} \|^{-1}
\end{eqnarray*}
where the last inequality follows from the Jensen inequality applied to $t \mapsto t^{-1}$. 
\end{proof}

\subsection{Proof of Theorems \ref{th-main critical exponent}}

Let us fix a discrete Zariski dense subgroup $\Gamma$ of $G$, and a Zariski dense normal subgroup $\Gamma' \triangleleft \Gamma$. Assume moreover that $\Gamma'$ is co-amenable in $\Gamma$, and let $\Theta$ be a right-invariant mean on $\Gamma'\backslash \Gamma$. Let $\phi\in C(\Gamma')_+^*$ be fixed. 

\medskip

We have seen in Section \ref{sec-PS measures} that there exists $\mu \in M_\phi(\delta_\phi(\G'), \G')$. Let $\mu^\Theta$ be the $\Theta$-average of $\mu$ given by Lemma \ref{lem:average-density}. We define
$$
\Gamma_+ = \left \{ \gamma\in \Gamma \; ; \; ||\mu^\Theta_{\gamma o}||\geq 1\right \}.
$$
It follows from (\ref{eq:big-mass}) that for all $\gamma\in \G$, we have $\gamma\in \Gamma_+$ or $\gamma^{-1}\in \Gamma_+$. Therefore, for all $s\in \mathbb R$,
$$
\sum_{\gamma\in \Gamma} e^{-s \phi(a(o, \gamma o))} \leq 2 \sum_{\gamma\in \Gamma_+} e^{-s \phi(a(o, \gamma o))} \leq \sum_{\gamma \in \Gamma_+} ||\mu^\Theta_{\gamma o}||e^{-s \phi(a(o, \gamma o))}.
$$

Moreover, by Proposition \ref{pr-Roblin combinatorial}, the critical exponent of $\displaystyle s\mapsto \sum_{\gamma \in \Gamma_+} ||\mu^\Theta_{\gamma o}||e^{-s \phi(a(o, \gamma o))}$ is at most $\delta_\phi(\Gamma')$. Therefore we have
$$
\delta_\phi(\Gamma) \leq \delta_\phi(\Gamma'),
$$
which concludes the proof of Theorem \ref{th-main critical exponent}.

\bibliographystyle{alpha}

\end{document}